\documentclass[12pt, twoside, leqno]{article}

\usepackage{amsmath}
\usepackage{amsfonts,amsthm,amssymb,eucal}
\usepackage{indentfirst}

\usepackage{enumerate}

\newtheorem{theorem}{Theorem}[section]
\newtheorem{lemma}[theorem]{Lemma}

\theoremstyle{definition}
\newtheorem*{definition}{Definition}
\newtheorem{remark}{Remark}

\numberwithin{equation}{section}

\newcommand{\Z}{\mathbb{Z}}
\newcommand{\N}{\mathbb{N}}
\newcommand{\R}{\mathbb{R}}
\newcommand{\Pl}{\mathcal{P}}

\begin{document}

\baselineskip=17pt

\title{On algebraic integers in short intervals\\
	and near  smooth curves}
\author{Friedrich G{\"o}tze\\
University of Bielefeld\\
Bielefeld, Germany\\
E-mail: goetze@math.uni-bielefeld.de
\and
Anna Gusakova\\
University of Bielefeld\\
Bielefeld, Germany\\
E-mail: agusakov@math.uni-bielefeld.de}

\date{}

\maketitle

\renewcommand{\thefootnote}{}



\footnote{Supported by SFB-701, Bielefeld University (Germany).}

\renewcommand{\thefootnote}{\arabic{footnote}}
\setcounter{footnote}{0}

\begin{abstract}
In 1970 A. Baker and W. Schmidt introduced regular systems of numbers and vectors, showing that the set of real algebraic numbers forms a regular system on any fixed interval. This fact was used to prove several important results in the metric theory of transcendental numbers. In this paper the concept of a regular system is applied  to the set of  algebraic integers $\alpha$ of height $\leq Q$ in intervals of length depending on $Q$.
\end{abstract}

\section{Introduction}

Many problems in the theory of Diophantine approximation are related to the distribution of algebraic numbers and algebraic integers \cite{Cassels57, Schmidt80}. In this paper we wish to investigate the distribution of algebraic integers on the real line and the
distribution of the points with algebraically conjugate integer coordinates in the Euclidean plane.

Let $P(t)=a_nt^n+\ldots +a_1t+a_0$, $a_i\in\Z$ be a polynomial with integer coefficients of degree $\deg P = n$. By the height of polynomial $P$ we mean the value $H(P)=\max\limits_{0\leq j\leq n}{|a_j|}$.

Let us consider an irreducible polynomial $P\in\Z[t]$ with coprime coefficients. The roots of this polynomial are algebraic numbers $\alpha$ of degree $n$ and height $H(\alpha)=H(P)$. When $a_n=1$, the roots of the polynomial $P(t) = t^n+a_{n-1}t^{n-1}+\ldots+a_1t+a_0$ are called algebraic integers $\alpha$ of degree $n$ and height $H(\alpha)=H(P)$. Let $\#\,S$ denote the cardinality of a finite set $S$ and $\mu_k\, D$ denote the Lebesgue measure of a measurable set $D\subset \R^k$, $k\in\N$. We define the following class of polynomials:
\[
\Pl_n(Q)=\left\{ P\in\Z[t]:\, \deg P\leq n,\, H(P)\leq Q\right\}.
\]
We emphasize that we restrict our attention to the case when $Q>Q_0$ is a sufficiently large integer. Furthermore, we will denote by $c_j>0$, $j\in \N$ positive real numbers independent of $H(P)$ and $Q$.

The first part of this paper is devoted to the study of one-dimensional case, namely algebraic integers. Over the last 20 years, new results providing a deeper insights into the distribution of algebraic numbers have been obtained. In particular, lower and upper bounds for the distances between algebraically conjugate numbers and the roots of different integer polynomials were obtained in the papers \cite{BeresnevichBernikGoetze10, BudarinaGoetze13, BugeaudMignotte04, Evertse04}.

Let us consider an interval $I\subset\left[-\frac12;\frac12\right]$ of length $|I|=c_1 Q^{-1}$.
It is of interest to know whether an interval  $I$  of this type contains  algebraic numbers $\alpha$ of degree $\deg\alpha\leq n$ and height $H(\alpha)\leq Q$. In case of positive answer we are also interested in finding lower bound for the number of such algebraic $\alpha\in I$. These problems were solved for $n=3$ in the paper of V.\,Bernik, N.\,Budarina and H.\,O'Donnell \cite{BerBudOdonn12} and a general result for an arbitrary $n$ was proved by V.\,Bernik and F.\,G\"otze \cite{BernikGoetze15}. The result of the paper \cite{BernikGoetze15} states that for any integer $Q\ge 1$ there exists an interval $I$ of length $|I|=\frac12\, Q^{-1}$, which doesn't contain any algebraic number $\alpha$ of an arbitrary degree and height $H(\alpha)\leq Q$. On the other hand, for $Q>Q_0$ sufficiently large any interval $I$ of length $|I| \ge c_1\, Q^{-1}$ contains at least $c_2\,Q^{n+1}|I|$ real algebraic numbers $\alpha$ of degree $\deg\alpha\leq n$ and height $H(\alpha)\leq Q$. Furthermore, these algebraic numbers form a regular system \cite{Bugeaud04}.

Our purpose is to obtain a similar result in the case of algebraic integers.

\begin{theorem}\label{theorem0}
For any integer $Q\ge 1$ there exists an interval $I$ of length $|I|= \frac 12\cdot Q^{-1}$ which doesn't contain algebraic integers $\alpha$ of height $H(\alpha)\leq Q$ and arbitrary degree $n\ge 2$.
\end{theorem}

It is easily seen that Theorem \ref{theorem0} follows from the results presented in
\cite{BernikGoetze15}, since algebraic integers form a subset of the set of  algebraic numbers.

\begin{theorem}\label{theorem1}
Let the constant $c_3$ and the number $Q>Q_0$ be sufficiently large. Then there exists a constant $c_4$ such that any interval $I$ of length $|I|= c_3\, Q^{-1}$ contains at least $c_4\,Q^n|I|$ real algebraic integers $\alpha$ of degree $\deg\alpha= n$, $n\ge 2$ and height $H(\alpha)\leq Q$.
\end{theorem}

\begin{remark}
It should be mentioned that condition $I\subset\left[-\frac12;\frac12\right]$ is not essential to the proof and can be dropped (see \cite{Beresnevich99, Bugeaud03} for more details).
\end{remark}
\begin{remark}
Another way to stating Theorem \ref{theorem1} is to say that the set of real algebraic integers of degree $n$ forms a regular system.
\end{remark}

\begin{definition}
Let $\Gamma$ be a countable set of real numbers and
$N:\Gamma\to \R^{+}$ be a positive-valued function. The pair $(\Gamma, N)$ is called a \emph{regular system} if there exists a constant $c_5 = c_5(\Gamma,N)>0$ such that for every interval $I\subset\R$ the following property is satisfied: for a sufficiently large number $T_0=T_0(\Gamma,N,I)>0$ and an arbitrary integer $T>T_0$ there exist $\gamma_1,\gamma_2,\ldots,\gamma_t \in \Gamma\cap I$ satisfying
\begin{align*}
1)\qquad&N(\gamma_i)\le T, \qquad 1\le i\le t,\\
2)\qquad&|\gamma_i-\gamma_j| > T^{-1}, \qquad 1\le i<j\le t,\\
3)\qquad&t > c_5\,T|I|.
\end{align*}
\end{definition}

A simple example of a regular system is the set of non-zero rational numbers $p/q$ together with the function 	$N(p/q):=q^{2}$. Similarly, the set of real algebraic numbers $\alpha$ of degree $n$ forms a regular system with respect to the function $N(\alpha)=\left(\frac{H(\alpha)}{\left(1+|\alpha|\right)^n}\right)^{n+1}$ and the set of real algebraic integers $\alpha$ of degree $n$ forms a regular system with respect to the function $N(\alpha)=\left(\frac{H(\alpha)}{\left(1+|\alpha|\right)^{n-1}}\right)^{n}$ (see \cite{BakerSchmidt70, Beresnevich99, Bugeaud03}). The interest of Theorem \ref{theorem1} is that in contrast to the result \cite{Bugeaud03} it allows one to clarify the relation between parameter $T_0$ and the length of the interval $I$.

We emphasize that the results mentioned above yield many interesting applications. For example, regular systems of algebraic numbers are used to obtain lower bounds for the Hausdorff dimension of various algebraic number sets \cite{BakerSchmidt70, DickinsonDodson} and to prove Khinchine-type theorems in the case of divergence \cite{Beresnevich99, BudDikBer, BudZor}.

In the second part of our paper we proceed with the study of two-dimensional analogue of Theorem \ref{theorem1}. An interesting result related to the distribution of points with algebraically conjugate coordinates in the Euclidean plane was obtained in the papers \cite{BernikGoetzeKukso14, BernikGoetzeGusakova16}. Let us consider a rectangle $E = I_1\times I_2$, where $I_1$, $I_2$ are
intervals of lengths $|I_1| = Q^{-s_1}$, $|I_2| = Q^{-s_2}$ for $0<s_1+s_2<1$.
Furthermore, from now on we make the assumption:
\[
E\cap \left\{(x,y)\in\R^2:\quad |x-y|\leq \varepsilon\right\}=\varnothing,
\]
where $\varepsilon > 0$ is a sufficiently small constant. Since the distance between algebraically conjugate numbers is bounded below \cite{BugeaudMignotte04, Evertse04}, this condition is not particularly restrictive, but it will simplify our argument. We call a point $(\alpha,\beta)$ an {\it algebraic point} if $\alpha$ and $\beta$ are algebraically conjugate numbers, and an {\it algebraic integer point} if $\alpha$ and $\beta$ are algebraically conjugate integers. In the paper \cite{BernikGoetzeGusakova16} it is shown that for $Q>Q_0$ any rectangle $E$ of size $\mu_2\,E = Q^{-s_1-s_2}$, $0<s_1+s_2<1$ contains at least $c_6\,Q^{n+1}\mu_2\, E$ algebraic points $(\alpha, \beta)$ of degree $\deg\alpha=\deg\beta \leq n$, $n\ge 2$ and height $H(\alpha)=H(\beta)\leq Q$.

We prove that the similar estimate holds in case of algebraic integer points.

\begin{theorem}\label{theorem2}
For any rectangle $E= I_1\times I_2$ of size $\mu_2\,E=|I_1|\cdot|I_2|=Q^{-s_1-s_2}$, $0<s_1+s_2<1$ there exists a constant $c_7$ such that rectangle $E$ contains at least $c_7\,Q^n\mu_2 E$  algebraic integer points $(\alpha, \beta)$ of  degree $\deg\alpha=\deg\beta = n$, $n\ge 4$ and height $H(\alpha)=H(\beta)\leq Q$ for $Q>Q_0$.
\end{theorem}

\begin{remark}
It should be noted, that the position of the rectangle $E$ is assumed to be fixed, namely the midpoint $(d_1,d_2)$ of the rectangle $E$ is independent of $Q$. Therefore, the values $c_7$ and $Q_0$ may depend on $d_1$ and $d_2$.
\end{remark}

This theorem deals with simple figure like rectangle, but it allows one to obtain the analogous estimates in the case of more complicated shapes. In particular, a number of interesting problems arise when distribution of algebraic points in a certain neighborhoods of smooth curves is investigated \cite{Huxley}. Let us mention several recent results in this area. Upper and lower bounds of the same order for the number of rational points near smooth curves have been obtained in the papers \cite{BereDickVel} and \cite{VaugVel}. The paper \cite{BernikGoetzeKukso14} from 2014 presents lower estimate for the number of algebraic points of arbitrary degree in neighborhoods of smooth curves.

Our main theorem is a restatement of the results of the paper \cite{BernikGoetzeKukso14} in terms of algebraic integers.

\begin{theorem}
Let $y=f(x)$ be a continuous differentiable function on an interval $J=[a,b]$ such that $\sup\limits_{x\in J}{|f'(x)|}<\infty$. Denote by $L_J(Q,\lambda)$ the following set:
\[
L_J(Q,\lambda)=\left\{(x,y)\in\R^2:x_1\in J, \left|y-f(x)\right|<c_8\,Q^{-\lambda}\right\},
\]
for $0<\lambda<\frac12$. Then for $Q>Q_0(n,J,f,\lambda)$ there exists a constant $c_9$ such that the set $L_J(Q,\lambda)$ contains at least $c_9\,Q^{n-\lambda}$ algebraic integer points $(\alpha,\beta)$ of degree $\deg\alpha=\deg\beta = n$, $n\ge 4$ and of height $H(\alpha)=H(\beta)\leq Q$.
\end{theorem}

\begin{proof}
We give only the main ideas of the proof. For more details we refer the reader to \cite{BernikGoetzeGusakova16}.

Let us consider a graph of the function $y=f(x)$ and the strip $L_J(Q,\lambda)$ for a fixed $0<\lambda<\frac12$. Divide the strip $L_J(Q,\lambda)$ into segments
\[
T_i=\left\{(x,y)\in\R^2:x\in J_i,\quad |y-f(x)|\leq Q^{-\lambda}\right\},
\]
where $J_i=[x_{i-1},x_{i}]$, $x_i=x_{i-1}+c_{10}\,Q^{-\lambda}$, $x_0=a$ and $1\leq i\leq m$. It is easy to check that $m >c_{11}\,Q^{\lambda}$ for $Q>Q_0$. Let $\bar{f}_i=\frac12\cdot\left(\max\limits_{x\in J_i}{f(x)}+\min\limits_{x\in J_i}{f(x)}\right)$. Consider the rectangles
\[
E_i=\left\{(x,y)\in\R^2: x\in J_i,\quad \left|y-\bar{f}_i\right|\leq c_{12}\,Q^{-\lambda}\right\},
\]
where $c_{12}$ are so chosen that $E_i\subset T_i$.

From Theorem \ref{theorem2} it follows that every rectangle $E_i$, $i=\overline{1,m}$ contains at least $c_{13}\,Q^{n-2\lambda}$ algebraic integer points of degree $n$ and height at most $Q$. Since $m>c_{11}\,Q^{\lambda}$, there must be at least $c_9\,Q^{n-\lambda}$ algebraic integer points $(\alpha,\beta)\in L_J(Q,\lambda)$.
\end{proof}

\section{Auxiliary statements}

In this section we have compiled some lemmas which will be used to prove Theorem \ref{theorem1} and Theorem \ref{theorem2}. The first paper discussing approximation by algebraic integers was written by H.\,Davenport and W.M\,Schmidt \cite{DavenportSchmidt}. Recently, their approach has been further developed by Y.\,Bugeaud \cite{Bugeaud03}. In our paper we are going to apply some of his ideas. The main geometric ingredient is Minkowski's theorems from the geometry of numbers.

\begin{lemma}[Minkowski's 2nd theorem on successive minima]\label{lm_Minkowski}
Let $K$ be a bounded central symmetric convex body in $\R^n$ with  successive minima $\tau_1,\ldots,\tau_n$.
Then
\[
\frac{2^n}{n!}\leq\tau_1\tau_2\ldots\tau_nV(K)\leq 2^n.
\]
\end{lemma}

The best general references here are \cite[pp. 203]{Cassels97}, \cite[pp. 59]{Gruber}.

\begin{lemma}[Bertrand postulate]\label{lm_Bertran}
For any integer $n \ge 2$ there exists a prime $p$ such that $n<p<2n$.
\end{lemma}

Proved by P. Chebyshev in 1850 (see for instance \cite[Theorem 2.4]{Nesterenko}).

\begin{lemma}[Eisenstein criterion]\label{lm_Eisenstein}
Let $P(t)=a_nt^n+\ldots +a_1t+a_0$ denote a polynomial with integer coefficients. If there exists a prime number $p$ such that:
\begin{equation}\label{eq0}
\begin{cases}
a_n\not\equiv 0 \mod{p},\\
a_i\equiv 0 \mod{p},\quad i=0,\ldots, n-1\\
a_0\not\equiv 0 \mod{p^2},
\end{cases}
\end{equation}
then $P$ is irreducible over the rational numbers.
\end{lemma}

For a proof see \cite{Eis}.

\begin{lemma}\label{lm_polynomial}
Consider a point $x\in\R$ and a polynomial $P$ with zeros $\alpha_1,\alpha_2,\ldots,\alpha_n$,
where $|x-\alpha_1| = \min\limits_{i} |x-\alpha_i|$. Then
\[
|x-\alpha_1| \le n|P(x)|\cdot|P'(x)|^{-1}.
\]
\end{lemma}

\begin{proof}
Considering the polynomial $P$ and its derivative $P'$ at the point $x$ we get
\[
|P'(x)||P(x)|^{-1}\leq\sum\limits_{i=1}^n{|x-\alpha_i|^{-1}}\leq n|x-\alpha_1|^{-1},
\]
which establishes the formula.
\end{proof}

\begin{lemma}[see {\cite{BernikGoetze15}}]\label{lm_BeGo}
Let $I\subset\R$ be the interval of length $|I| = c_{14}\,Q^{-1}$, where $c_{14}>c_0$. Denote by $\mathcal{L}_n^1 = \mathcal{L}^1_n(Q,\delta_0,I)$ the set of points $x\in I$
such that there exists a polynomial $P\in\mathcal{P}_n(Q)$ satisfying the following system of inequalities
\[
\begin{cases}
|P(x)| < Q^{-n},\\
|P'(x)| < \delta_0 Q.
\end{cases}
\]
Then $\mu_1\,\mathcal{L}^1_n < \textstyle\frac14\,|I|$ for $\delta_0 = \delta_0(n)>0$ sufficiently small and $Q>Q_0$.
\end{lemma}

\begin{remark}
It suffices to take $\delta_0(n)=2^{-n-8}n^{-2}$ (see \cite{BernikGoetze15} for more details).
\end{remark}

This lemma is base for the proof of Theorem \ref{theorem1}.

\begin{lemma}[see {\cite{BernikGoetzeGusakova16}}]\label{lm_BeGoGu}
Let $E=I_1\times I_2$ be a rectangle with midpoint $(d_1,d_2)$ and sides $|I_i|=Q^{-s_i}$, $0<s_1+s_2<1$. Given positive $v_1$, $v_2$ satisfying $v_1+v_2=n-1$, let $\mathcal{L}^2_n = \mathcal{L}^2_n(Q,\delta_0,E,v_1,v_2)$ be set of points $(x,y)\in E$, such that there exists a polynomial $P\in\mathcal{P}_n(Q)$ satisfying the following system of inequalities
\begin{equation}\label{eq000}
\begin{cases}
|P(x)| < h_1\,Q^{-v_1},\quad |P(y)| < h_2\,Q^{-v_2},\\
\min\limits_i{\{|P'(x)|,|P'(y)|\}} < \delta_0 Q,
\end{cases}
\end{equation}
where $h_i=\left(\left(|d_i|+1\right)^{n+1}-1\right)|d_i|^{-1}$, $i=1,2$. Then $\mu_2\,\mathcal{L}^2_n < \textstyle\frac14\, \mu_2\,E$ for $\delta_0 = \delta_0(n, d_1, d_2)>0$ sufficiently small and $Q>Q_0$. 
\end{lemma}

\begin{remark}
An easy computations shows that for every point $(x,y)\in E$ and for all polynomials $P\in\mathcal{P}_{n}(Q)$ we have the following estimates:
\[
|P(x)|<h_1\,Q,\quad |P(y)|<h_2\,Q.
\]
Hence the values $v_1$ and $v_2$ lie between $-1$ and $n$.
\end{remark}

\begin{remark}
It it easily seen (for example from Lemma \ref{lm_polynomial}) that for a fixed polynomial $P$ the set of points $(x,y)\in\R^2$ satisfying the system (\ref{eq000}) is contained in a rectangle $\sigma_P=J_1\times J_2$ of measure $\mu_2\sigma_P\leq\frac14\,\mu_2\,E$ (see \cite{BernikGoetzeGusakova16}). If $I_1\subset J_1$ or $I_2\subset J_2$,  we consider the rectangle $I_1 \times J_2$ or $J_1\times I_2$ instead of the rectangle $\sigma_P$ to estimate the measure of $\mathcal{L}^2_n$.
\end{remark}

\section{Proof of Theorem \ref{theorem1}}

Let $\mathcal{L}^1_{n-1} = \mathcal{L}^1_{n-1}(Q,\delta_0,I)$ be the set of $x\in I$ such that there exists a polynomial $P\in\mathcal{P}_{n-1}(Q)$ satisfying the inequalities:
\begin{equation}\label{eq1}
\begin{cases}
|P(x)| < Q^{-n+1},\\
|P'(x)| < \delta_0\,Q.
\end{cases}
\end{equation}
From Lemma \ref{lm_BeGo} it follows that the measure of the set $\mathcal{L}_{n-1}^1$ can be estimated as
\[
\mu\,\mathcal{L}^1_{n-1}\leq\frac14\,|I|,
\]
for $Q>Q_0$ and $\delta_0<2^{-n-7}(n-1)^{-2}$.

Let us consider the set $B^1=I\backslash \mathcal{L}^1_{n-1}$. Since for any $x\in I$ there exists a polynomial $P\in\mathcal{P}_{n-1}(Q)$ satisfying $|P(x)| < Q^{-n+1}$ we conclude that for any $x_0\in B^1$ and the polynomial $P\in\mathcal{P}_{n-1}(Q)$, the system of inequalities
\[
\begin{cases}
|P(x_0)| < Q^{-n+1},\\
|P'(x_0)| \ge \delta_0\,Q,
\end{cases}
\]
is satisfied and $\mu_1\,B^1 \ge \frac34\,|I|$.

Consider an arbitrary point $x_0\in B^1$ and examine successive minima
$\tau_1,\ldots,\tau_n$ of the compact convex set $K$ defined by inequalities
\begin{equation}\label{eq2}
\begin{cases}
|a_{n-1}x_0^{n-1}+\ldots+a_1x_0+a_0|\leq Q^{-n+1},\\
|(n-1)a_{n-1}x_0^{n-2}+\ldots+2a_2x_0+a_1| \leq Q,\\
|a_{n-1}|,\ldots,|a_2|\leq Q.
\end{cases}
\end{equation}
Let $\tau_1\leq\delta_0$. Then for $\delta_0$ sufficiently small there exists a non-zero polynomial $P_0\in\mathcal{P}_{n-1}(Q)$ satisfying the inequalities
\[
\begin{cases}
|P_0(x_0)| \leq\delta_0\,Q^{-n+1} < Q^{-n+1},\\
|P_0'(x_0)| \leq \delta_0\,Q,\\
H(P_0) \leq Q.
\end{cases}
\]
This contradicts the fact that $x_0\in B^1=I\backslash \mathcal{L}^1_{n-1}$, following us to conclude that $\tau_{n-1}\ge\ldots\ge\tau_1>\delta_0$. Since the volume $V(K)$ of the compact convex set $K$ is equal to $2^n$, we deduce, from Lemma \ref{lm_Minkowski}, that $\tau_1\ldots\tau_n\leq 1$ and, hence, that $\tau_{n}\leq \delta_0^{-n+1}$. Therefore we can choose $n$ linearly independent polynomials with integer coefficients $P_i(t)=a_{i,n-1}t^{n-1}+\ldots+a_{i,1}t+a_{i,0}$, $1\leq i\leq n$ satisfying the system of inequalities
\begin{equation}\label{eq3}
\begin{cases}
|P_i(x_0)|\leq \delta_0^{-n+1}\,Q^{-n+1},\\
|P_i'(x_0)| \leq \delta_0^{-n+1}\,Q,\\
|a_{i,j}|\leq \delta_0^{-n+1}\,Q,\quad 2\leq j\leq n-1.
\end{cases}
\end{equation}
Applying  well-known estimates from the geometry of numbers (see \cite[pp. 219]{Cassels97}) for the polynomials $P_i$ we obtain the inequality
\[
\Delta = \det|(a_{i,j-1})^n_{i,j=1}|\leq n!.
\]
Moreover, from Lemma \ref{lm_Bertran} it follows that there exists a prime $p$, which doesn't divide  $\Delta$ and satisfies
\begin{equation}\label{eq4}
n!<p< 2n!.
\end{equation}

Our next goal is to construct the irreducible monic polynomial of degree $n$ using polynomials $P_i$. Consider the following system of linear equations in $n$ variables  $\theta_1,\ldots,\theta_n$:
\begin{equation}\label{eq5}
\begin{cases}
x_0^n+p\,\sum\limits_{i=1}^{n}{\theta_i P_i(x_0)}=p(n+1)\delta_0^{-n+1}\,Q^{-n+1},\\
nx_0^{n-1}+p\,\sum\limits_{i=1}^{n}{\theta_i P_i'(x_0)}=p\,Q+ p\,\sum\limits_{i=1}^{n}{|P_i'(x_0)|},\\
\sum\limits_{i=1}^{n}{\theta_i a_{i,j}}=0,\quad 2\leq j\leq n-1.
\end{cases}
\end{equation}

In order to find the determinant $\hat{\Delta}$ of this system, it is convenient to transform it as follows. Multiply the $k$-th equation, where $k=3, \ldots, n$, by $p\cdot x_0^{k-1}$ and subtract it from the first equation of the system (\ref{eq5}). Similarly, multiply the $k$-th equation, where $k=3,\ldots, n$, by $p\cdot (k-1)x_0^{k-2}$ and subtract it from the second equation. After making these transformations the determinant $\hat{\Delta}$ may be written as
\[
\hat{\Delta}=p^2\cdot
\begin{vmatrix}
a_{1,1}x_0+a_{1,0} & \dots & a_{n,1}x_0+a_{n,0} \\
a_{1,1} & \dots &  a_{n,1} \\
\vdots & \ddots & \vdots \\
a_{1,n-1} & \dots & a_{n,n-1}
\end{vmatrix}
\]
Since the polynomials $P_i$ are linearly independent, we conclude that $\hat{\Delta}=p^2\Delta\neq 0$. Hence, there exists a unique solution
$(\theta_1,\ldots,\theta_n)$ of the system (\ref{eq5}). 

For integers $k_1,\ldots,k_n$ consider the following construction, which is a polynomial of degree $n$ with integer coefficients:
\[
P(t)=t^n+p\cdot\sum\limits_{i=1}^n{k_iP_i(t)}=t^n+p\cdot(a_{n-1}t^{n-1}+\ldots+a_1t+a_0),
\]
where $a_j=\sum\limits_{i=1}^{n}{k_ia_{i,j}}$, $0\leq j\leq n-1$ and $k_i$, $1\leq i\leq n$ satisfies
\begin{equation}\label{eq6}
|\theta_i - k_i|\leq 1.
\end{equation}
 
 We next show that there exists a suitable combinations of the coefficients $k_i$ such that the polynomial $P$ is irreducible. From inequality (\ref{eq6}) we have two possible values for every $k_i$, which will be denoted by $k_i^1$ and $k_i^2=k_i^1+1$. Therefore, by Lemma \ref{lm_Eisenstein}, it suffices to show that we can choose $k_i$ such that all $a_j$ satisfy (\ref{eq0}). It is easily seen that the first and the second conditions of (\ref{eq0}) hold for any $k_i$. It remains to show that $a_0=k_1a_{1,0}+\ldots+k_na_{n,0}$ isn't divisible by $p$. Since $p$ doesn't divide $\Delta$, there exists a number $1\leq i\leq n$ such that $a_{i,0}$ is not divisible by $p$ and hence either $a_0^1=k_1a_{1,0}+\ldots+a_{i,0}k_i^1+\ldots+a_{n,0}k_n$ or $a_0^2=k_1a_{1,0}+\ldots+a_{i,0}k_i^2+\ldots+a_{n,0}k_n$ is also not divisible by $p$. Therefore, choosing $k_i$ in this manner yields an irreducible polynomial $P$.

We now proceed to estimate $|P(x_0)|$, $|P'(x_0)|$ and $H(P)$. Combining (\ref{eq3}) and (\ref{eq6}) with the system of equations (\ref{eq5}) we obtain the following inequalities.

From the first equation of the system it follows that
\begin{equation}\label{eq7}
p\delta_0^{-n+1}\,Q^{-n+1}\leq|P(x_0)|\leq p(2n+1)\delta_0^{-n+1}\,Q^{-n+1}.
\end{equation}
Similarly, from the second equation of the system we have
\begin{equation}\label{eq8}
p\,Q\leq |P'(x_0)|\leq (p+2pn\delta_0^{-n+1})\,Q,
\end{equation}
and the remaining equations of the system give
\begin{equation}\label{eq9}
|a_j|\leq n\delta_0^{-n+1}\,Q,\quad 2\leq j\leq n-1.
\end{equation}
Finally, applying (\ref{eq7})---(\ref{eq9}) and the inequality $|x_0|\leq \frac12$ yields the following estimates for the coefficients $a_1$ and $a_0$:
\begin{multline}\label{eq10}
|a_1|\leq |P'(x_0)|+n|x_0|^{n-1}+\sum\limits_{j=2}^{n-1}{j|x_0|^{j-1}|a_j|}\leq(p+2pn\delta_0^{-n+1})\,Q\\
  +\left(n\delta_0^{-n+1}\sum\limits_{k=1}^{n-1}\textstyle\frac{k+1}{2^{k}}\right)\,Q\leq\left(p+\left(2pn+3n\right)\delta_0^{-n+1}\right)\,Q,
\end{multline}
\begin{multline}\label{eq11}
|a_0|\leq |P(x_0)|+|x_0|^n+|a_1x_0|+\sum\limits_{j=2}^{n}{|x_0|^j|a_j|}\leq \textstyle\frac12p\,Q\\
+\left(\textstyle\frac12p+\left(pn+\textstyle\frac32n\right)\delta_0^{-n+1}\right)\,Q+\textstyle\frac12n\delta_0^{-n+1}\,Q\leq\left(p+\left(pn+4n\right)\delta_0^{-n+1}\right)\,Q. 
\end{multline}
From the estimates (\ref{eq9})---(\ref{eq11}) and the inequality (\ref{eq4}) we conclude that
\begin{equation}\label{eq12}
H(P)\leq 2n!\left(2n\delta_0^{-n+1}+1\right)\,Q=Q_1.
\end{equation}

Consider the roots $\alpha_1,\ldots,\alpha_n$ of the polynomial $P$, where $|x_0-\alpha_1|=\min\limits_{i}{|x_0-\alpha_i|}$. In view of Lemma \ref{lm_polynomial}, the following estimate holds
\begin{equation}\label{eq12_0}
|x_0-\alpha_1|\leq n|P(x_0)||P'(x_0)|^{-1}.
\end{equation}
Substituting inequalities (\ref{eq7}) and (\ref{eq8}) into (\ref{eq12_0}) we obtain
\begin{equation}\label{eq13}
|x_0-\alpha_1|\leq n(2n+1)\delta_0^{-n+1}\,Q^{-n}=c_{15}\,Q^{-n}.
\end{equation}

If $\alpha_1$ is a complex root of the polynomial $P$, then its conjugate is also a root of the polynomial $P$. Hence, by (\ref{eq12}), (\ref{eq13}) and well-known estimates for the roots of the polynomial $P$, namely $|\alpha_i|\leq H(P)+1$, $1\leq i\leq n$ (see \cite[Theorem 1.1.2]{Praslov}), we deduce that
\[
|P(x_0)|=\prod\limits_{i=1}^n|x_0-\alpha_i|\leq c_{15}^2Q^{-2n}\cdot\left(2+2n!\left(2n\delta_0^{-n+1}+1\right)\,Q\right)^{n-2}.
\]
This inequality contradicts (\ref{eq7}) for $Q>Q_0$. Thus, $\alpha_1$ is real.

Finally, take a maximal system of real algebraic integers $\Gamma=\left\{\gamma_1,\ldots,\gamma_m\right\}$ such that $|\gamma_i-\gamma_j|>c_{15}\,Q^{-n}$, $1\leq i\neq j\leq m$. Let us show that for any point $x_0\in B^1$ there exists an algebraic number $\gamma \in \Gamma$ such that $|x_0-\gamma|\leq 2c_{15}\,Q^{-n}$. According to the above arguments and (\ref{eq13}) for any point $x_0\in B^1$ there exists a real algebraic integer $\alpha_1\in I$ such that $|x_0-\alpha_1|\leq c_{15}\,Q^{-n}$. If $\alpha_1\in\Gamma$, then we can take $\gamma=\alpha_1$, otherwise, there exists $\gamma_i\in\Gamma$ such that $|\alpha_1-\gamma_i|\leq c_{15}Q^{-n}$ and, hence,
\[
|x_0-\gamma_i|\leq |x_0-\alpha_1|+|\alpha_1-\gamma_i|\leq 2 c_{15}Q^{-n}.
\]
In this case, we can take $\gamma=\gamma_i$. Therefore, $B^1$ is contained in a union $\bigcup\limits_{i=1}^m\left\{x\in I:|x-\gamma_i|\leq 2 c_{15}\,Q^{-n}\right\}$ and
\[
4 mc_{15}\,Q^{-n}\ge\mu_1\,\left(\bigcup\limits_{i=1}^m\left\{x\in I:|x-\gamma_i|\leq 2 c_{15}\,Q^{-n}\right\}\right)\ge \mu_1\,B^1\ge\textstyle\frac34\,|I|.
\]
This inequality implies that the number of algebraic integers $\alpha\in I$, $\deg\alpha=n$, $H(\alpha)\leq Q_1$ is no smaller then
\[
m>\textstyle\frac{3}{16}c_{15}^{-1}\,Q^n|I|=\textstyle\frac{3}{16}c_{15}^{-1}\left(2n!\left(2n\delta_0^{-n+1}+1\right)\right)^{-1}\,Q_1^n|I|=c_4\,Q_1^n|I|
\]
for $Q_1>Q_0$ and the proof is complete.

From the proof of Theorem \ref{theorem1} it follows, that the set of algebraic integers of degree $n$ forms a regular system with respect to the function $N(\alpha)=\left(\frac{H(\alpha)}{\left(1+|\alpha|\right)^{n-1}}\right)^{n}$ and $T_0=c_{16}|I|^{-n}$, where the constant $c_{16}$ independent of $|I|$.

\section{Proof of Theorem \ref{theorem2}}

The proof of Theorem \ref{theorem2} follows by the same method as the proof of Theorem \ref{theorem1}, but it contains some non-trivial elements which require special attention.

The proof of Theorem \ref{theorem2} apply Lemma \ref{lm_BeGoGu}, which is two-dimensional analogue of Lemma \ref{lm_BeGo}. Given positive $v_1$ and $v_2$ satisfying the condition $v_1+v_2=n-2$, let us consider a system of inequalities
\begin{equation}\label{eq14}
\begin{cases}
|P(x)| < \hat{h}_1\,Q^{-v_1},\quad |P(y)| <  \hat{h}_2\,Q^{-v_2},\\
\min\limits{\{|P'(x)|,|P'(y)|\}} < \delta_0\, Q,
\end{cases}
\end{equation}
where $\hat{h}_i=\max\left\{\left(\left(|d_i|+1\right)^n-1\right)|d_i|^{-1}, \frac14|d_1-d_2|^{-2}\right\}$, $i=1,2$. Lemma \ref{lm_BeGoGu} implies that the measure of the set $\mathcal{L}^2_{n-1}=\mathcal{L}^2_{n-1}(Q,\delta_0,E,v_1,v_2)$ of points $(x,y)\in E$ such that there exists a polynomial $P\in\mathcal{P}_{n-1}(Q)$ satisfying (\ref{eq14}) can be estimated as
\[
\mu_2\, \mathcal{L}^2_{n-1}\leq\textstyle\frac14\, \mu_2\,E
\]
for $Q>Q_0$ and $\delta_0$ sufficiently small.

It is easy to check using for example Minkowski's theorem on linear forms \cite[pp. 73]{Cassels97}, that for any point $(x,y)\in E$ there exists a polynomial $P\in\mathcal{P}_{n-1}(Q)$ satisfying $|P(x)| < \hat{h}_1\,Q^{-v_1}$ and $|P(y)| < \hat{h}_2\,Q^{-v_2}$. From this it follows that for any point $(x,y)\in B^2=E\backslash \mathcal{L}^2_{n-1}$ we may choose a polynomial $P\in\mathcal{P}_{n-1}(Q)$ such that the system
\[
\begin{cases}
|P(x)| < \hat{h}_1\,Q^{-v_1},\quad |P(y)| < \hat{h}_2\,Q^{-v_2},\\
|P'(x)|\ge \delta_0\,Q,\quad |P'(y)|\ge \delta_0\,Q,
\end{cases}
\]
holds and $\mu_2\, B^2 \ge \frac34\,\mu_2\,E$.

As in the proof of Theorem \ref{theorem1} consider an arbitrary point $(x_{0},y_{0})\in B^2$ and examine the successive minima $\tau_1,\ldots,\tau_n$ of the compact convex set $K$ defined by
\[
\begin{cases}
\left|a_{n-1}x_{0}^{n-1}+\ldots+a_1x_{0}+a_0\right|\leq  \hat{h}_1\,Q^{-v_1},\\
\left|a_{n-1}y_{0}^{n-1}+\ldots+a_1y_{0}+a_0\right|\leq  \hat{h}_2\,Q^{-v_2},\\
\left|(n-1)a_{n-1}x_{0}^{n-2}+\ldots+2a_2x_{0}+a_1\right| \leq Q,\\
\left|(n-1)a_{n-1}y_{0}^{n-2}+\ldots+2a_2y_{0}+a_1\right| \leq Q,\\
|a_i|\leq Q,\quad 4\leq i\leq n-1.
\end{cases}
\]
Assume $\tau_1\leq\delta_0$. Then for $\delta_0$ sufficiently small there exists a polynomial $P_0\in\mathcal{P}_{n-1}(Q)$ satisfying the inequalities
\[
\begin{cases}
|P_0(x_{0})| < \delta_0\hat{h}_1\,Q^{-v_1}<\hat{h}_1\,Q^{-v_1},\quad |P_0(y_{0})| < \delta_0\hat{h}_2\,Q^{-v_2}<\hat{h}_2\,Q^{-v_2},\\
|P_0'(x_{0})|<\delta_0\,Q,\quad |P_0'(y_{0})|<\delta_0\,Q,\\
H(P_0) < Q.
\end{cases}
\]
contrary to the fact that $(x_{0},y_{0})\in B^2$. Thus, $\tau_1>\delta_0$. This fact and estimate $V(K)>2^n$ allows us to use Lemma \ref{lm_Minkowski}, namely inequality $\tau_1\ldots\tau_n\leq 1$, to conclude that $\tau_{n}\leq \delta_0^{-n+1}$. Hence, there exist $n$ linearly independent
polynomials with integer coefficients $P_i(t)=a_{i,n-1}t^{n-1}+\ldots+a_{i,1}t+a_{i,0}$, $1\leq i\leq n$ satisfying the inequalities
\begin{equation}\label{eq15}
\begin{cases}
|P_i(x_{0})|\leq \delta_0^{-n+1}\hat{h}_1\,Q^{-v_1},\quad |P_i(y_{0})|\leq \delta_0^{-n+1}\hat{h}_2\,Q^{-v_2}\\
|P_i'(x_{0})| \leq \delta_0^{-n+1}\,Q,\quad |P_i'(y_{0})| \leq \delta_0^{-n+1}\,Q,\\
|a_{i,j}|\leq \delta_0^{-n+1}\,Q,\quad 4\leq j\leq n-1.
\end{cases}
\end{equation}
Analysis similar to that in the proof of Theorem \ref{theorem1} shows that there exists a prime $p$ which doesn't divide $\Delta=\det|(a_{i,j-1})^n_{i,j=1}|$ and satisfies
\begin{equation}\label{eq16}
n! <p< 2n!.
\end{equation}

Next, let us consider a system of linear equations in $n$ variables $\theta_1,\ldots,\theta_n$
\begin{equation}\label{eq17}
\begin{cases}
x_{0}^n+p\,\sum\limits_{i=1}^{n}{\theta_i P_i(x_{0})}=p(n+1)\delta_0^{-n+1}\hat{h}_1\,Q^{-v_1},\\
y_{0}^n+p\,\sum\limits_{i=1}^{n}{\theta_i P_i(y_{0})}=p(n+1)\delta_0^{-n+1}\hat{h}_2\,Q^{-v_2},\\
nx_{0}^{n-1}+p\,\sum\limits_{i=1}^{n}{\theta_i P_i'(x_{0})}=pQ+ p\,\sum\limits_{i=1}^{n}{|P_i'(x_{0})|},\\
ny_{0}^{n-1}+p\,\sum\limits_{i=1}^{n}{\theta_i P_i'(y_{0})}=pQ+ p\,\sum\limits_{i=1}^{n}{|P_i'(y_{0})|},\\
\sum\limits_{i=1}^{n}{\theta_i a_{i,j}}=0,\quad 4\leq j\leq n-1.
\end{cases}
\end{equation}
Our goal is to show that the determinant $\hat{\Delta}$ of this system is not vanish. Let us transform the system (\ref{eq17}) as follows.
Multiply the $k$-th equation, where $k=5, 6, \ldots, n$, by $p\cdot x_{0}^{k-1}$ ( respectively by $p \cdot y_{0}^{k-1}$) and subtract it from the first (respectively the second) equation of the system (\ref{eq17}). Similarly, multiply the $k$-th equation, where $k=5, 6, \ldots, n$, by $p\cdot (k-1)x_{0}^{k-2}$ (respectively  by $p \cdot (k-1)y_{0}^{k-2}$) and subtract it from the third (respectively the fourth) equation. After these transformations the determinant of system (\ref{eq17}) may be written as
\[
p^4\cdot
\begin{vmatrix}
a_{1,3}x_{0}^3+a_{1,2}x_0^2+a_{1,1}x_0+a_{1,0} & \dots & a_{n,3}x_{0}^3+a_{n,2}x_0^2+a_{n,1}x_0+a_{n,0}\\
a_{1,3}y_{0}^3+a_{1,2}y_0^2+a_{1,1}y_0+a_{1,0} & \dots & a_{n,3}y_{0}^3+a_{n,2}y_0^2+a_{n,1}y_0+a_{n,0}\\
3a_{1,3}x_{0}^2+2a_{1,2}x_0+a_{1,1} & \dots & 3a_{n,3}x_{0}^2+2a_{n,2}x_0+a_{n,1}\\
3a_{1,3}y_{0}^2+2a_{1,2}y_0+a_{1,1} & \dots & 3a_{n,3}y_{0}^2+2a_{n,2}y_0+a_{n,1}\\
a_{1,4} & \dots & a_{n,4}\\
\vdots & \ddots & \vdots \\
a_{1,n-1} & \dots & a_{n,n-1}
\end{vmatrix}
\]
We proceed to show that $\hat{\Delta}$ is equal to $\Delta$ up to a multiple depending only on $x_0$, $y_0$ and $p$. Multiply the third (respectively the fourth) row by $\frac13 x_{0}$ (respectively by $\frac13 y_{0}$) and subtract it from the first (respectively the second) row. Then subtracting the first (respectively the third) row from the second (respectively the fourth) row gives:
\[
\hat{\Delta}=\textstyle\frac{p^4(y_{0}-x_{0})^2}{9}\cdot\begin{vmatrix}
a_{1,2}x_{0}^2+2a_{1,1}x_{0}+3a_{1,0} & \dots & a_{n,2}x_{0}^2+2a_{n,1}x_{0}+3a_{n,0} \\
a_{1,2}(y_{0}+x_{0})+2a_{1,1} & \dots & a_{n,2}(y_{0}+x_{0})+2a_{n,1} \\
3a_{1,3}x_{0}^2+2a_{1,2}x_{0}+a_{1,1} & \dots & 3a_{n,3}x_{0}^2+2a_{n,2}x_0+a_{n,1} \\
3a_{1,3}(y_{0}+x_{0})+2a_{1,2}& \dots & 3a_{n,3}(y_{0}+x_{0})+2a_{n,2} \\
a_{1,4} & \dots & a_{n,4}\\
\vdots & \ddots & \vdots \\
a_{1,n-1} & \dots & a_{n,n-1}
\end{vmatrix}.
\]
Now let us subtract the second row multiplied by $x_{0}$ from the first row and the fourth row multiplied by $\frac12$ from the third row. Then subtract the third row multiplied by $\frac{y_{0}+x_{0}}{x_{0}^2}$ from the fourth row, and finally subtract the fourth row multiplied by $x_{0}y_{0}$, $y_{0}+x_{0}$ and $\frac32x_{0}-\frac12y_{0}$ from the first, the second and the third row respectively. Consequently we obtain the inequality
\[
\hat{\Delta}=p^4(y_{0}-x_{0})^4\cdot
\begin{vmatrix}
a_{1,0} & \dots & a_{n,0}\\
\vdots & \ddots & \vdots \\
a_{1,n-1} & \dots & a_{n,n-1}
\end{vmatrix}
=p^4(y_{0}-x_{0})^4\Delta >0,
\]
becouse the polynomials $P_i$, $1\leq i\leq n$ are linearly independent and $|y_{0}-x_{0}|>\varepsilon>0$.
Hence, the system (\ref{eq17}) has a unique solution $(\theta_1,\ldots,\theta_n)$. Moreover, there exist integers $k_1,\ldots,k_n$ satisfying
\begin{equation}\label{eq19}
|\theta_i - t_i|\leq 1,\quad i=1,\ldots,n,
\end{equation}
such that the following polynomial with integer coefficients:
\[
P(t)=t^n+p\cdot\sum\limits_{i=1}^n{k_iP_i(t)}=t^n+p\cdot(a_{n-1}t^{n-1}+\ldots+a_1t+a_0),
\]
where $a_j=\sum\limits_{i=1}^{n}{k_ia_{i,j}}$, $0\leq j\leq n-1$ is irreducible. This follows by the same arguments as in the previouse section.

Let us estimate the values $|P(x_{0})|$, $|P(y_{0})|$, $|P'(x_{0})|$ and $|P'(y_{0})|$. From inequalities (\ref{eq15}), 
(\ref{eq19}) and the first four equations of the system (\ref{eq17}) we see that:
\begin{align}
p\delta_0^{-n+1}\hat{h}_1\,Q^{-v_1}&\leq|P(x_{0})|\leq p(2n+1)\delta_0^{-n+1}\hat{h}_1\,Q^{-v_1}, \label{eq20}\\ 
p\delta_0^{-n+1}\hat{h}_2\,Q^{-v_2}&\leq|P(y_{0})|\leq p(2n+1)\delta_0^{-n+1}\hat{h}_2\,Q^{-v_2}, \label{eq21}
\end{align}
\begin{align}
p\,Q\leq |P'(x_{0})|\leq \left(p+2pn\delta_0^{-n+1}\right)\,Q, \label{eq22}\\
p\,Q\leq |P'(y_{0})|\leq \left(p+2pn\delta_0^{-n+1}\right)\,Q. \label{eq23}
\end{align}
Finally, we need to estimate the height $H(P)$.
By the fourth to  $n$-th equations of the system (\ref{eq17}) and inequalities (\ref{eq15}), (\ref{eq19}), we have
\begin{equation}\label{eq24}
|a_j|\leq n\delta_0^{-n+1}\,Q,\quad 4\leq j\leq n-1.
\end{equation}
The only point remaining concerns the estimation of $|a_j|$, $0\leq j\leq 3$. By (\ref{eq20}) -- (\ref{eq24}) and the inequalities $|x_{0}|\leq |d_1|+\frac12$, $|y_{0}|\leq |d_2|+\frac12$, for $Q>Q_0$ we have
\begin{multline*}
\left|a_3x_{0}^3+a_2x_{0}^2+a_1x_{0}+a_0\right|\leq |P(x_{0})|+\sum\limits_{j=4}^{n-1}{|x_0|^j|a_j|}+|x_0|^n \\
 <3pn\delta_0^{-n+1}\hat{h}_1\,Q^{-v_1}+\left(n\delta_0^{-n+1}\sum\limits_{j=4}^n{\left(|d_1|+\textstyle\frac12\right)^j}\right)\,Q<4pn\delta_0^{-n+1}\hat{h}_1\,Q,
\end{multline*}
and, similarly, 
\begin{equation*}
\left|a_3y_{0}^3+a_2y_{0}^2+a_1y_{0}+a_0\right|<4pn\delta_0^{-n+1}\hat{h}_2\,Q.
\end{equation*}
Then
\begin{multline*}
|3a_3x_{0}^2+2a_2x_{0}+a_1|\leq |P'(x_{0})|+\sum\limits_{j=4}^{n-1}{j|x_0|^{j-1}|a_j|}+n|x_0|^{n-1}\\
<\left(p+2pn\delta_0^{-n+1}\right)\,Q+\left(n\delta_0^{-n+1}\sum\limits_{j=4}^n{j\left(|d_1|+\textstyle\frac12\right)^{j-1}}\right)\,Q\\
<\left(p+2pn\delta_0^{-n+1}+n^2\hat{h}_1\delta_0^{-n+1}\right)\,Q,
\end{multline*}
and, similarly, 
\begin{equation*}
|3a_3y_{0}^2+2a_2y_{0}+a_1|\leq\left(p+2pn\delta_0^{-n+1}+n^2\hat{h}_2\delta_0^{-n+1}\right)\,Q.
\end{equation*}
We emphasize that in order to simplify equations we do not care about the accuracy of the constants. Consider the following system of linear equations for $a_0$, $a_1$, $a_2$ and $a_3$:
\begin{equation}\label{eq26}
\begin{cases}
a_3x_{0}^3+a_2x_{0}^2+a_1x_{0}+a_0=l_{1},\\
a_3y_{0}^3+a_2y_{0}^2+a_1y_{0}+a_0=l_{2},\\
3a_3x_{0}^2+2a_2x_{0}+a_1=l_{3},\\
3a_3y_{0}^2+2a_2y_{0}+a_1=l_{4}.
\end{cases}
\end{equation}

According to the above computations the determinant of the system (\ref{eq26}) does not vanish. Thus, the system has a unique solution, which may be found by using Cramer's rule. Combining this with estimates above one can easily verify:
\[
|a_j|<c_{17}n\delta_0^{-n+1}\,Q,\quad 0\leq j\leq 3,
\]
where $c_{17}=2^8p\varepsilon^{-3}\left(\hat{h}_1+\hat{h}_2\right)\left(\max\{|d_1|,|d_2|\}\right)^3$. Applying (\ref{eq16}) and (\ref{eq24}) now yields the following estimate:
\[
H(P)<c_{18}n\delta_0^{-n+1}\,Q=Q_1,
\]
where $c_{18}=\max\{1,c_{17}\}$.

Consider the roots $\alpha_1,\ldots,\alpha_n$ of the polynomial $P$, where $|x_{0}-\alpha_1|=\min\limits_{i}{|x_{0}-\alpha_i|}$ and let $\beta_1,\ldots,\beta_n$ be a permutation of these roots such that $|y_{0}-\beta_1|=\min\limits_{i}{|y_{0}-\beta_i|}$. By Lemma \ref{lm_polynomial} and estimates (\ref{eq20}) -- (\ref{eq22}), we have
\[
\begin{cases}
|x_{0}-\alpha_1|< n(2n+1)\delta_0^{-n+1}\hat{h}_1\,Q^{-v_1-1}=c_{19}\hat{h}_1\,Q^{-v_1-1},\\
|y_{0}-\beta_1|< n(2n+1)\delta_0^{-n+1}\hat{h}_2\,Q^{-v_2-1}=c_{19}\hat{h}_2\,Q^{-v_2-1}.
\end{cases}
\]
For $Q>Q_0$, the roots $\alpha_1$ and $\beta_1$ are real, as is easy to check.

Let $\Gamma=\left\{(\alpha_1, \beta_1),\ldots,(\alpha_m, \beta_m)\right\}$ be a maximal system of real algebraic integer points such that
\[
|\alpha_i-\alpha_j|>c_{19}\hat{h}_1\,Q^{-v_1-1}\quad \text{or}\quad |\beta_i-\beta_j|>c_{19}\hat{h}_2\,Q^{-v_2-1},\qquad 1\leq i\neq j\leq m.
\]
This follows by the same method as in the previous section that for any point $(x_{0},y_{0})\in B^2$ there exists an algebraic integer point $(\alpha_i,\beta_i)\in\Gamma$ satisfying
\[
|x_{0}-\alpha_i|< 2c_{19}\hat{h}_1\,Q^{-v_1-1},\quad |y_{0}-\beta_i|< 2c_{19}\hat{h}_2\,Q^{-v_2-1}.
\]
This implies the following covering:
\[
B^2\subset\bigcup_{i=1}^m\left\{(x,y)\in E:|x-\alpha_i|< 2c_{19}\hat{h}_1\,Q^{-v_1-1}, |y-\beta_i|< 2c_{19}\hat{h}_2\,Q^{-v_2-1}\right\},
\]
where
\[
m>\textstyle\frac{3}{64}\cdot c_{19}^{-2}\hat{h}_1^{-1}\hat{h}_2^{-1}\,Q^n\mu_2\,E=c_7\,Q_1^n\mu_2\,E,
\]
which finishes the proof.

\subsection*{Acknowledgments}

The authors wish to express their thanks to Prof. V. Bernik for suggesting the problem and for numerous enlightening conversations.

This research was partly supported by SFB-701.

\newpage
\noindent Anna Gusakova,\\
Faculty of Mathematics, University of Bielefeld, \\
PO Box 100131, 33501 Bielefeld, Germany\\
E-mail: agusakov@math.uni-bielefeld.de\\
\\
Friedrich G{\"o}tze,\\
Faculty of Mathematics, University of Bielefeld, \\
PO Box 100131, 33501 Bielefeld, Germany\\
E-mail: goetze@math.uni-bielefeld.de
\end{document}